\documentclass[10pt,a4paper,oneside]{amsproc}

\usepackage{latexsym, amsthm}
\usepackage{amsfonts, amsmath, amssymb,comment}
\usepackage{euscript,mathrsfs}
\usepackage{url}
\usepackage{array}
\usepackage[a4paper]{geometry}
\usepackage{graphics}
\usepackage[usenames]{color}
\usepackage[all]{xy}
\usepackage{graphicx}
\usepackage{boxedminipage}

\newtheorem{theorem}{Theorem}[section]

\newtheorem{conjecture}{Conjecture}
\newtheorem{corollary}[theorem]{Corollary}
\newtheorem{question}[theorem]{Question}
\newtheorem{definition}[theorem]{Definition}
\newtheorem{remark}[theorem]{Remark}

\newtheorem{lemma}[theorem]{Lemma}

\newtheorem{proposition}[theorem]{Proposition}

\newcommand{\codim}{\mathrm{codim}}
\newcommand{\B}{\mathcal{B}}

\newcommand{\Fqt}{\mathbb{F}_{q^2}}
\newcommand{\Fq}{\mathbb{F}_q}

\newcommand{\Sing}{\mathrm{Sing}}

\newcommand{\C}{\mathcal{C}}

\def\Fq{{\mathbb F}_q}
\def\AA{{\mathbb A}}
\def\FF{{\mathbb F}}
\def\PP{{\mathbb P}} 
\newcommand{\V}{{\mathsf{V}}}

\newcommand{\X}{\mathcal{X}}

\begin{document}

\title[Intersection of Hermitian threefolds and cubic threefolds]{Maximum number of points on an intersection of a cubic threefold and a non-degenerate Hermitian threefold}

\author{Mrinmoy Datta}
\address{Department of Mathematics, \newline \indent
Indian Institute of Technology Hyderabad, Kandi, Sangareddy, Telangana, India}
\email{mrinmoy.datta@math.iith.ac.in}
\thanks{The first named author is partially supported by the grant DST/INT/RUS/RSF/P-41/2021 from the 
Department of Science and Technology, Govt. of India and the grant SRG/2021/001177 from Science and Engineering Research Board, Govt. of India}

\author{Subrata Manna}
\address{Department of Mathematics, \newline \indent
Indian Institute of Technology Hyderabad, Kandi, Sangareddy, Telangana, India}
\email{subrata147314@gmail.com}
\thanks{The second named author is partially supported by a doctoral fellowship from Council of Scientific and Industrial Research, Govt. of India}


\begin{abstract}   
It was conjectured by Edoukou in 2008 that a non-degenerate Hermitian threefold in $\PP^4 (\Fqt)$ has at most $d(q^5+q^2) + q^3 + 1$ points in common with a threefold of degree $d$ defined over $\Fqt$. He proved the conjecture for $d=2$. In this paper, we show that the conjecture is true for $d = 3$ and $q \ge 7$. 
\end{abstract}

\date{}
\maketitle

\section{Introduction}

Hermitian varieties over finite fields are fascinating objects in the theory of algebraic geometry, finite geometry, coding theory, and combinatorics. Bose and Chakravarti first studied these varieties in \cite{BC}. The popularity of Hermitian varieties grew among coding theorists since the non-degenerate Hermitian varieties admit a very high number of rational points. Furthermore, Hermitian varieties come with fascinating point-line incidence structures, which are attractive to mathematicians working in finite geometry and combinatorics. From a coding theoretic perspective, the following question has been studied in \cite{SoT, E, E1, BD, BDH} among other articles. 

\begin{question}\label{q1}
    Let $V_{m-1}$ be a non-degenerate Hermitian variety in $\PP^m (\Fqt)$. Determine the maximum possible number of $\Fqt$-rational points common to a hypersurface of degree $d$ in $\PP^m$ defined over $\Fqt$ and $V_{m-1}$.
\end{question}

The case when $d=1$ is completely settled in \cite{BC, C}. When $m=2$, one could derive from using B\'{e}zout's Theorem that a curve of degree $d$ and a non-degenerate Hermitian curve intersect at a maximum of $d (q+1)$ many $\Fqt$-rational points. It is fairly easy to produce a plane curve of degree $d$ defined over $\Fqt$ that intersects the Hermitian curve at $d(q+1)$ points. However, the question becomes much more difficult as $m$ increases. The case $m=3$, to our knowledge, was first studied by S{\o}rensen, in his Ph.D. thesis \cite{SoT}. He conjectured that a non-degenerate Hermitian surface in $\PP^3(\Fqt)$ has at most $d (q^3 + q^2 - q) + q + 1$ many $\Fqt$-rational points in common with a surface in $\PP^3$ of degree $d$ defined over $\Fqt$. Even though the conjecture was formulated in the early 1990s, the first breakthrough was made by Edoukou in 2007. In \cite{E}, he proved that the conjecture is true for $d = 2$. S{\o}rensen's conjecture was finally proved in a series of two papers \cite{BD, BDH}. It is now natural to look at the case when $m=4$. Edoukou, in \cite{E1}, determined the answer to Question \ref{q1} for $d=2$ and $m = 4$. Besides,  Edoukou proposed the following conjecture: 

\begin{conjecture}\cite[Conjecture 1]{E1}\label{EdC}
Let $F \in \Fqt [x_0, \dots, x_4]$ be a homogeneous polynomial of degree $d$ and $V_3$ be a non-degenerate Hermitian threefold in $\PP^4 (\Fqt)$. If $d \le q$, then 
$$|V(F)(\Fqt) \cap V_3| \le d (q^5 + q^2) + q^3 + 1.$$
Moreover, this bound is attained by a threefold $V(F)$ in $\PP^4$, if and only if $V(F)$ is a union of $d$ non-tangent hyperplanes $\Sigma_1, \dots, \Sigma_d$ defined over $\Fqt$, such that $\Sigma_1, \dots, \Sigma_d$ contains a common plane $\Pi$ defined over $\Fqt$ and $\Pi$ intersects $V_3$ at a non-degenerate plane Hermitian curve. 
\end{conjecture}

As mentioned above, he proved Conjecture \ref{EdC} for $d=2$. The conjecture remained open to date. In this paper, we will prove the conjecture for $d = 3$ and $q \ge 7$. It is worth noting that Edoukou's proof of the conjecture for $d = 2$ depends on the well-known classification of quadrics in $\PP^4 (\Fqt)$. However, our approach is more algebro-geometric. We will in fact provide a proof of Edoukou's result using our method in Corollary \ref{Eprove}.

This paper is organized as follows. In Section \ref{sec:prel}, we recall various well-known properties of Hermitian varieties, and in particular non-degenerate Hermitian threefolds, revisit some preliminary notions from algebraic geometry and some bounds on the number of rational points on varieties defined over a finite field.  
In Section \ref{sec:singular} we take a step back to derive an upper bound for the number of $\Fqt$-rational points common to a degenerate Hermitian surface and a surface of degree $d$ in $\PP^3 (\Fqt)$.   In Section \ref{sec:gen} we present some results for general values of $d \le q$ towards proving Edoukou's conjecture. Finally, in Section \ref{sec:cubic}, we prove our main result.

\section{Preliminaries}\label{sec:prel}
Let us fix $q$, a prime power. As usual, $\Fq$ and $\Fqt$ denote the finite fields with $q$ and $q^2$ elements, respectively. For $m \ge 0$, we denote by $\PP^m$, the projective space of dimension $m$ over the algebraic closure $\overline{\FF}_q$, while $\PP^m (\Fqt)$ will denote the set of all $\Fqt$-rational points on $\PP^m$. 
Similarly, $\AA^m$ and $\AA^m (\Fqt)$ will denote the affine space of dimension $m$ over $\overline{\FF}_q$ and the set of all $\Fqt$-rational points of $\AA^m$ respectively.  
Further, for a homogeneous polynomial $F \in \Fqt[x_0, \dots , x_m]$, we denote by $V(F)$, the set of zeroes of $F$ in $\PP^m$ and by $\V(F)=V(F)(\Fqt)$ the set of all $\Fqt$-rational points of $V(F)$. By an algebraic variety, we mean a set of zeroes of a certain set of polynomials in the affine space or projective space, depending on the context. In particular, an algebraic variety need not be irreducible. We remark that whenever we say that a variety is nonsingular, we will mean that the variety is nonsingular over $\overline{\FF}_q$. However, when we say that a hypersurface $V(F)$ is irreducible, we will mean that the polynomial $F$ is irreducible in the field of its definition. This section is divided into three subsections: in the first subsection, we recall several known facts about Hermitian varieties over finite fields; the second subsection is dedicated to revisiting some useful notions from basic algebraic geometry, while the third subsection concerns some known bounds on the number of rational points on varieties defined over a finite field. None of the results mentioned in this section is new, and their proof can be found in the indicated references.

\subsection{Geometry of Hermitian threefolds over finite fields}\label{herm}

To begin with, let us revisit the definition and several properties of Hermitian varieties (cf. \cite{BC, C}) that will be used in the latter part of this paper. 

\begin{definition}[\cite{BC}] \normalfont
Let $A = (a_{ij})$ ($0 \le i, j \le m$) be an $(m+1) \times (m+1)$ matrix with entries in $\Fqt$. We denote by $A^{(q)}$ the matrix whose $(i, j)$-th entry is given by $a_{ij}^q$. The matrix $A$ is said to be a \textit{Hermitian matrix} if $A \neq 0$ and $A^T = A^{(q)}$. A \textit{Hermitian variety} of dimension $m-1$, denoted by $V_{m-1}$, is the set of zeroes of the polynomial $x^T A x^{(q)}$ inside $\PP^m$, where $A$ is an $(m+1) \times (m+1)$ Hermitian matrix and $x = (x_0, \dots, x_m)^T$. From a finite-geometric perspective, the authors of \cite{BC} denoted by $V_{m-1}$ the $\Fqt$-rational points on the variety given by the equation $x^T A x ^{(q)} = 0$. Strictly speaking, from an algebro-geometric point of view, one should denote this set by $V_{m-1}(\Fqt)$. But to maintain the classical notation, we will stick to the symbol $V_{m-1}$ to denote the set of all $\Fqt$-rational points on $V_{m-1}$. The Hermitian variety is said to be \textit{non-degenerate} if $\mathrm{rank} \ A = m+1$ and \textit{degenerate} otherwise. 
\end{definition}

It was established in \cite[Equation 5.6]{BC} that if the rank of a Hermitian matrix is $r$, then by a suitable change of coordinates over $\Fqt$, one can describe the corresponding Hermitian variety by the zero set of the polynomial 
\begin{equation}\label{herm}
x_0^{q+1} + x_1^{q+1} + \dots + x_{r-1}^{q+1} = 0.
\end{equation}
We note that the polynomial $x_0^{q+1} + \dots + x_{r-1}^{q+1}$ is irreducible over the algebraic closure of $\Fq$ whenever $r \ge 3$. 
Thus, a Hermitian variety of rank at least three is always absolutely irreducible, i.e. the variety is irreducible over $\overline{\FF}_q$. Moreover, if $V_{m-1}$ is non-degenerate, then it is non-singular. As mentioned in the Introduction, the linear sections of Hermitian varieties are rather well-understood thanks to \cite{BC, C}. The following result is a straightforward consequence of \cite[Theorem 7.4]{BC} and \cite[Theorem 3.1]{C}. 

 \begin{theorem}\label{hyperplane}
     Let $V_{m-1}$ be a non-degenerate Hermitian variety in $\PP^m$ and $\Sigma$ be a hyperplane of $\PP^{m}$ defined over $\Fqt$. 
     \begin{enumerate}
     \item[(a)] \cite[Theorem 7.4]{BC} if $\Sigma$ is tangent to $V_{m-1}$ at a point $P \in V_{m-1}$, then 
     $\Sigma \cap V_{m-1}$ is a cone over a non-degenerate Hermitian variety $V_{m-3}$ contained in a hyperplane of $\Sigma$ with center at $P$. 
     \item[(b)] \cite[Theorem 3.1]{C} \footnote{Strictly speaking, the result by Chakravarti is about the section of non-degenerate Hermitian varieties with hyperplanes that are polar hyperplanes at external points. However, it is not difficult to show that any non-tangent hyperplane can be realized as a polar hyperplane to the non-degenerate Hermitian variety at an external point.} if $\Sigma$ is not a tangent to $V_{m-1}$, then $\Sigma \cap V_{m-1}$ is a non-degenerate Hermitian variety $V_{m-2}$ in $\Sigma$.
     \end{enumerate}
 \end{theorem}

\noindent We now describe the possibilities arising out of a line intersecting a Hermitian variety. 

\begin{lemma}\cite[Section 7]{BC}\label{line}
Any line in $\PP^m (\Fqt)$ satisfies precisely one of the following.
\begin{enumerate}
\item[(i)] The line intersects $V_{m-1}$ at precisely $1$ point. 
\item[(ii)] The line intersects $V_{m-1}$ at precisely $q+1$ points. 
\item[(iii)] The line is contained in $V_{m-1}$. 
\end{enumerate}
\end{lemma}

Let us summarize the known facts regarding the linear sections of non-degenerate Hermitian threefolds in $\PP^4$. From now onwards, we will denote by $V_3$ a non-degenerate Hermitian threefold in $\PP^4$. 
It is well-known that a non-degenerate Hermitian threefold does not contain a plane. One can derive this from \cite[Theorem 9.1]{BC} and also from a much stronger result \cite[Lemma 2.1]{HK}.  We have the following definitions thanks to Lemma \ref{line}.
 
\begin{definition}\label{lines} \normalfont
Let $\ell$ be a line in $\PP^4(\Fqt)$.  The line $\ell$ is called 
\begin{enumerate}
\item[(a)]  a \textit{tangent line} if it intersects $V_3$ at exactly $1$ point. 
\item[(b)]  a \textit{secant line} if it intersects $V_3$ at exactly $q+1$ points. 
\item[(c)] a \textit{generator} if it is contained in $V_3$. 
\end{enumerate}
\end{definition}

\begin{remark}\label{rem0}\normalfont
 In particular, if $\Sigma$ is a hyperplane in $\PP^4$ defined over $\Fqt$, and $V_3$ denotes the non-degenerate Hermitian threefold in $\PP^4$, then the following two possibilities occur:
\begin{enumerate}
    \item[(a)] if $\Sigma$ is not a tangent to $V_3$, then $\Sigma \cap V_3$ is a non-degenerate Hermitian surface  in $\Sigma$.
    \item[(b)] if $\Sigma$ is a tangent to $V_3$ at $P \in V_3$, then $\Sigma \cap V_3$ is a cone over a nonsingular plane Hermitian curve $V_1$ with center at $P$. If $V_1$ is contained in a plane $\Pi \subset \Sigma$, then it is understood that $P \not\in \Pi$. 
 \end{enumerate}   
    
\noindent  Since any line $\ell$ with $P \in \ell \subseteq V_3$,  is contained in $T_P(V_3)$, the tangent hyperplane to $V_3$ at $P$, it follows from part (b) that the non-degenerate Hermitian threefold contains $q^3 + 1$ lines defined over $\Fqt$ passing through $P$. 
  Also as a consequence of \cite[Theorem 9.2]{BC}, the non-degenerate Hermitian threefold $V_3$ in $\PP^4$ contains exactly $(q^5 + 1)(q^3 + 1)$ lines defined over $\Fqt$.  
\end{remark}
From \cite[Section 7]{BC}, we know that the intersection of a Hermitian variety with a linear subspace is a Hermitian variety. In particular, if $\Pi$ is a plane in $\PP^4$ defined over $\Fqt$, then  $V_3 \cap \Pi$ is also a Hermitian variety. There are three possibilities :

\begin{enumerate}
    \item[(a)] \textit{$\Pi \cap V_3$ is a non-degenerate Hermitian curve:} in this case, $|\Pi \cap V_3| = q^3 + 1$.
    \item[(b)] \textit{$\Pi \cap V_3$ is a union of $q+1$ lines through a common point:} in this case $|\Pi \cap V_3| = q^3 +q^2 + 1$.
    \item[(c)] \textit{$\Pi \cap V_3$ is a line:} in this case $|\Pi \cap V_3| = q^2 + 1$.
\end{enumerate}

\noindent In particular, if $\Pi$ is a plane defined over $\Fqt$ that is contained in the tangent hyperplane $T_P (V)$ to $V_3$ at a point $P \in V_3$ and $P \not\in \Pi$, then $\Pi \cap V_3$ is a non-degenerate Hermitian curve. 

For a fixed line $\ell$ in $\PP^4$ defined over $\Fqt$,  We will denote by $\B (\ell)$, the \textit{book of planes containing $\ell$}, that is, $$\B (\ell) := \{ \Pi : \ell \subset \Pi, \Pi \ \text{is a plane in} \ \PP^4 \ \text{defined over } \ \Fqt\}.$$ It is easy to see that $|\B(\ell)| = q^4 + q^2 + 1$. Moreover, if $\Sigma$ is a hyperplane in $\PP^4$ defined over $\Fqt$ containing $\Pi$, we denote by $\B_{\Sigma} (\ell)$ the \textit{book of planes in $\Sigma$ containing $\ell$}, that is, $$\B_{\Sigma} (\ell) := \{ \Pi \in \B(\ell):  \Pi \subset \Sigma\}.$$ It is fairly easy to show that $|\B_{\Sigma} (\ell)| = q^2 + 1$. 

If $P, Q \in V_3$ are distinct points and $\ell$ is a generator of $V_3$ through $P$ and $Q$, then the tangent spaces $T_P (V_3)$ and $T_Q(V_3)$ intersect at a plane $\Pi$ such that $\Pi \cap V_3 = \ell$. Also, if a plane $\Pi$ intersects $V_3$ at $q+1$ lines through $P$, then $\Pi \subset T_P (V_3)$. 

As mentioned in the Introduction, Edoukou \cite{E1} dealt with the sections of quadric threefolds and non-degenerate Hermitian threefolds. We record his results here for future references.

\begin{theorem}\cite[Theorem 4.1]{E1}\label{EdTQ}
    Let $\mathcal{Q}$ be a quadric threefold in $\PP^4$ defined over $\Fqt$. Then $|Q(\Fqt) \cap V_3| \le 2q^5 + q^3 + 2q^2 + 1$.
\end{theorem}

Besides proving the above-mentioned theorem, Edoukou also provided the five highest possible number of points of intersection between a quadric and a non-degenerate Hermitian threefold. We present an excerpt of this result \cite[Theorem 4.4]{E1} below. For future reference, we first introduce the following definition.

\begin{definition}\normalfont
Let $V_3$ be a non-degenerate Hermitian threefold in $\PP^4$. A quadric threefold $Q$ in $\PP^4$ defined over $\Fqt$ is said to be of 
\begin{enumerate}
    \item[(a)] \textbf{Type I} with respect to $V_3$, if $Q$ is a union of two distinct non-tangent hyperplanes $\Sigma_1$ and $\Sigma_2$ such that $\Sigma_1\cap\Sigma_2$ intersects $V_3$ at a nonsingular Hermitian plane curve.
    \item[(b)] \textbf{Type II} with respect to $V_3$, if $Q$ is a union of two distinct non-tangent hyperplanes $\Sigma_1$ and $\Sigma_2$ such that $\Sigma_1 \cap \Sigma_2$ intersects $V_3$ at a union of $q+1$ lines passing through a common point.
    \item[(c)] \textbf{Type III} with respect to $V_3$, if $Q$ is a union of two hyperplanes $\Sigma_1$ and $\Sigma_2$, where $\Sigma_1$ is a tangent to $V_3$ and $\Sigma_2$ is a non-tangent and $\Sigma_1 \cap \Sigma_2$ intersects $V_3$ at a nonsingular plane Hermitian curve. 
\end{enumerate}
\end{definition}
In addition to Theorem \ref{EdTQ}, Edoukou proved the following result.     
\begin{proposition}\cite[Theorem 4.4]{E1}\label{ed1}
Let $Q$ be a quadric threefold in $\PP^4$ defined over $\Fqt$. Then
\begin{equation*}
|Q(\Fqt) \cap V_3| 
\begin{cases}
= 2(q^5 + q^2) + q^3 + 1, \ & \mathrm{if} \ $Q$ \ \mathrm{is \ of \ Type \ I} \\
= 2q^5 + q^3 + q^2 + 1, \ & \mathrm{if} \ $Q$ \ \mathrm{is \ of \ Type \ II} \\
= 2q^5 +  2q^2 + 1, \  & \mathrm{if} \ $Q$ \ \mathrm{is \ of \ Type \ III} \\\
\le 2q^5 + q^2  + 1, \  & \mathrm{otherwise}.
\end{cases}
\end{equation*}
    \end{proposition}

We conclude the current subsection by recalling the S{\o}rensen's bound from \cite{BDH}.

\begin{theorem}\cite{BDH}\label{SoB}
Let $F \in \Fqt[x_0, x_1, x_2, x_3]$ be a homogeneous polynomial of degree $d$ and let $V_2$ be a non-degenerate Hermitian surface in $\PP^3 (\Fqt)$. If $d \le q$, then
$$|V(F)(\Fqt) \cap V_2| \le d(q^3 + q^2 - q) + q + 1.$$
Moreover, this bound is attained by $V(F)$ if and only if $V(F)$ is a union of $d$ planes $\Pi_1, \dots, \Pi_d$ such that $\Pi_1, \dots, \Pi_d$ intersect at a line $\ell$ that is a secant to $V_2$. 
\end{theorem}

\noindent In particular, if $V(F)$ is not a union of planes, then the inequality in the above theorem is strict. We will use this observation in Proposition \ref{dgen}.

\subsection{Preliminaries from algebraic geometry}
Here we recall various basic results from algebraic geometry that will be instrumental for proving our main theorem. We will make use of the notions of dimension, degree and singularity of a variety, as can be found in standard textbooks of Algebraic Geometry, for example, the book of Harris \cite{H}. We begin with the following definitions. A variety is said to be of \textit{pure dimension} or \textit{equidimensional} if all the irreducible components of the variety have equal dimension. Further, two equidimensional varieties $X, Y \subset \PP^m$ are said to \textit{intersect properly} if $\codim (X \cap Y) = \codim \ X + \codim \  Y.$ It turns out that good upper bounds for the number of rational points on varieties defined over a finite field, depend on the degree and dimension of the variety. Because of this reason, the following proposition from \cite{H} will be indispensable for us.

\begin{proposition}\cite[Cor. 18.5]{H}\label{deg}
Let $X$ and $Y$ be equidimensional varieties that intersect properly in $\PP^m$. Then 
$\deg (X \cap Y) \le \deg X \deg Y$.
\end{proposition}

\begin{proposition}\label{coprime}
Let $F, G \in \Fqt[x_0, x_1, \dots, x_m]$ be nonconstant homogeneous polynomials having no common factors. Then 
\begin{enumerate}
\item[(a)] $V(F)$ and $V(G)$ intersect properly. 
\item[(b)] $V(F) \cap V(G)$ is equidimensional of dimension $m-2$.
\item[(c)] $\deg (V(F) \cap V(G)) \le \deg F \deg G$. 
\end{enumerate}
\end{proposition}

\begin{proof}
Part (a) is well-known in Algebraic Geometry and is generally proved using Krull's principal ideal theorem. Part (b) is proved using Macaulay's unmixedness theorem (See \cite[Chapter 7, Theorem 26]{ZS}). Part (c) follows straightaway from Proposition \ref{deg}. 
\end{proof}

\noindent We conclude this subsection by recording a very important observation that follows from the two propositions above. 

\begin{remark}\label{rem1}\normalfont
    Suppose $V$ is a Hermitian variety in $\PP^m (\Fqt)$ of rank at least $3$. Since a Hermitian variety is absolutely irreducible, if $X$ is any hypersurface in $\PP^m$ defined over $\Fqt$ with $\deg X \le q$, then $X$ and $V$ intersect properly. As a consequence, $\deg (X \cap V) \le  (q+1)\deg X.$
\end{remark}

\subsection{Basic upper bounds}
In this subsection, we recall some well-known upper bounds on the number of rational points on varieties defined over a finite field $\Fq$ with given degrees and dimensions. We begin with a result by J. -P. Serre \cite{S} that is undoubtedly one of the fundamental results. This was independently proved by S{\o}rensen \cite{So}. 
 
\begin{theorem}[Serre's inequality]\label{serre}
Let $ F \in \Fq[x_0, x_1, \dots, x_m]$ be a nonconstant homogeneous polynomial of degree $d \le q$. Then 
$$|\V(F)| \le dq^{m-1} + p_{m-2},$$
where $p_{m-2} = 1 + q + \dots + q^{m-2}$. Moreover,
equality holds if and only if $V(F)$ is a union of $d$ hyperplanes defined over $\Fq$ all containing a common linear subspace of codimension $2$. 
\end{theorem}
\noindent In particular, any plane curve defined over $\Fq$ of degree $d \le q$ has at most $dq+1$ many $\Fq$-rational points. Furthermore, a plane curve of degree $d$ that admits exactly $dq+1$ many $\Fq$-rational points is a union of $d$ lines passing through a common point. Among various other known bounds for plane curves defined over finite fields, we would make use of the following result \cite[Corollary 2.5]{AP} by Aubry and Perret. 

\begin{theorem}\cite[Corollary 2.5]{AP}\label{ap}
Let $\C$ be a plane absolutely irreducible algebraic curve defined over $\Fq$. Then $|\#C(\Fq) - (q+1)| \le (d-1)(d-2)\sqrt{q}$, where $d = \deg \C$.  
\end{theorem}

For nonsingular curves, the above bound is known as \textit{Hasse-Weil bound}. However, the result by Aubry and Perret applies to a larger family of curves. To solve our problem, it is imperative that we consider bounds on varieties that are not necessarily hypersurfaces. To this end, we have the following result by Lachaud and Rolland. 

\begin{proposition}\cite[Prop. 2.3]{LR}\label{lac}
Let $X$ be an equidimensional projective (resp. affine) variety defined over a finite field $\Fq$. Further assume that $\dim X = \delta$ and $\deg X = d$.  Then
$$|X(\Fq)| \le d p_{\delta} \ \ \ \ \ \ (\mathrm{resp.} \ \  |X (\Fq)| \le d q^{\delta}),$$
where $p_{\delta} = 1 + q + \dots + q^{\delta}$.
\end{proposition}

The following result from \cite{LR} on varieties that are irreducible over $\Fq$, but not absolutely irreducible, will be very important. 

\begin{proposition}\cite[Prop. 3.8]{LR}\label{lac1}
Let $X$ be a projective variety defined over $\Fq$ that is irreducible over $\Fq$, but not absolutely irreducible. Then $X(\Fq) = \Sing (X) (\Fq)$, where $\Sing (X)$ is the subset of $X$ consisting of all the singular points of $X$ in $\overline{\FF}_q$. 
\end{proposition}
 
The above result shows, in particular, that if $X$ is a projective variety defined over $\Fq$, that is irreducible over $\Fq$ but not absolutely irreducible, then any $\Fq$-rational point on $X$ is a singular point of $X$. We conclude this section by stating a recent result that will be useful later.

\begin{theorem}\label{D}\cite[Theorem 3.1]{D}
    Let $Y \subset \PP^3$ be a surface of degree $d$ defined over $\Fq$ and $P \in Y(\Fq)$. Then one of the following holds: 
    \begin{enumerate}
        \item[(a)] $Y$ contains a plane defined over $\Fq$, 
        \item[(b)] $Y$ contains a cone over a plane curve defined over $\Fq$ with center at $P$, 
        \item[(c)] $\# \{\ell \subset \PP^3 \mid \ell \ \mathrm{is \ a \ line \ such \ that} \ P \in \ell \subset Y\} \le d(d - 1)$.
        \end{enumerate}
    \end{theorem}

\begin{remark}\label{remD}\normalfont
    We explain the application of Theorem \ref{D} in case of cubic surfaces. Suppose $Y$ is a cubic surface in $\PP^3$ that does not contain any plane defined over $\Fq$ and $P \in Y(\Fq)$. If $Y$ contains a cone over a plane curve $Y'$ defined over $\Fq$ with center at $P$, then $\deg Y' = 3$. Indeed, if $\deg Y' = 1$ or $2$, then $Y$ contains a plane defined over $\Fq$, which contradicts our assumption. Moreover, it is clear that $Y'$ must be irreducible over $\Fq$. The Theorem above thus implies that such a cubic surface $Y$ is either a cone over a plane irreducible cubic curve $Y'$ with center at $P$, or $Y$ may contain at most $6$ lines passing through $P$. 
\end{remark}

\section{Intersection of a singular Hermitian surface and a surface of degree $d$}\label{sec:singular}
Let $V_2'$ denote a singular Hermitian surface of rank $3$ in $\PP^3$. Geometrically, the Hermitian variety $V_2'$ is a cone over a plane nonsingular Hermitian curve $V_1$ with center at a point $P$. Suppose $F \in \Fqt[x_0, x_1, x_2, x_3]$ be a homogeneous polynomial of degree $d$. Then $V(F)$ is a surface of degree $d$ in $\PP^3$ defined over $\Fqt$.
We note that any plane that does not pass through $P$ intersects $V_2'$ at a non-degenerate Hermitian curve, and $V_2'$ is still a cone over that Hermitian curve with center at $P$. In this section, we are going to give an upper bound on the quantity $|V(F)(\Fqt) \cap V_2'|$ that will be instrumental in the proof of our main result. Throughout this section, we will assume that $d \le q$. Since our main purpose of this article is to work towards a proof of Conjecture \ref{EdC}, it is enough for us to assume that $d \le q$. 

\begin{lemma}\label{lem31}
 If $\Pi$ is a plane defined over $\Fqt$ in $\PP^3$ that does not pass through $P$ and $V(F)$ does not contain $\Pi$, then $|\V(F) \cap V_2' \cap \Pi| \le d(q+1)$.    
\end{lemma}

\begin{proof}
As noted above, the curve $V_2' \cap \Pi$ is a non-degenerate Hermitian curve in the plane $\Pi$ and, in particular, an irreducible plane curve of degree $q+1$. By hypothesis, the curve $V(F) \cap \Pi$ is a plane curve of degree $d$. The assertion now follows from B\'{e}zout's Theorem. 
\end{proof}

\begin{lemma}\label{lem32}
If $\Pi$ is a plane in $\PP^3$ defined over $\Fqt$ such that $P \not\in \Pi$ and $\Pi \subset V(F)$, then $|V(F)(\Fqt) \cap V_2' \cap \Pi^C| \le (d-1)(q+1)q^2$, where $\Pi^C = \PP^3 \setminus \Pi$.  
\end{lemma}

\begin{proof}
Since $\Pi \subset \V(F)$, it follows that $\V(F) \cap \Pi^C$ is an affine surface of degree at most $d-1$. Furthermore, $V_2' \cap \Pi^C$ is also an irreducible affine surface of degree $q+1$ contained in $\Pi^C$. It follows that $(V(F) \cap \Pi^C)$ and $(V_2' \cap \Pi^C)$ have no common components, and consequently, their intersection is a complete intersection of degree at most $(d-1)(q+1)$. Consequently, $\dim (V_2' \cap V(F) \cap \Pi^C) = 1$ and  
it follows from Proposition \ref{lac}, that $|V(F)(\Fqt) \cap V_2' \cap \Pi^C| \le (d-1)(q+1)q^2$. 
\end{proof}

We are ready to state and prove the main result of this section. 

\begin{theorem}\label{degen}
We have $|\V(F) \cap V_2'| \le d(q+1)q^2 + 1$.  
\end{theorem}

\begin{proof}

\textbf{Case 1: $V(F)$ does not contain any plane \textit{not} passing through $P$.} We may assume that $V(F)$ intersects $V_2'$ at an $\Fqt$-rational point $Q$ other than $P$. 
Let $\ell$ be a line in $\PP^3$ defined over $\Fqt$ that intersects $V_2'$  only at $Q$. 
It is evident that $P \not\in \ell$. 
Then there exists a unique plane, say $\Pi_0$, containing $\ell$ that passes through $P$. Note that $\Pi_0 \cap V_2' = \ell'$, where $\ell'$ is the line joining $P$ and $Q$. Indeed, the line $\ell' \subseteq \Pi_0 \cap V_2'$ since $P, Q \in \Pi_0 \cap V_2'$. 
We write $\B (\ell) = \{\Pi_0, \Pi_1, \dots, \Pi_{q^2} \}$, where $\B (\ell)$ denotes the set of all planes containing $\ell$.  It follows that,
\begin{equation}\label{eq31}
|\V(F) \cap V_2' \cap \Pi_0| \le q^2 + 1.
\end{equation}
Furthermore, the planes $\Pi_1, \dots, \Pi_{q^2}$ are planes that do not contain $P$. As a consequence, $\Pi_j$ intersects $V_2'$ at a nonsingular Hermitian curve for every $j = 1, \dots, q^2$. Since $\Pi_j \not\subseteq V(F)$, we see that $|V(F) \cap \Pi_j|$ is a plane curve of degree $d$ for all $j = 1, \dots, q^2$. From Lemma \ref{lem31}, we have
\begin{equation}\label{eq32}
|\V(F) \cap V_2' \cap \Pi_i| \le d(q+1) \ \text{for} \ i=1, \dots, q^2.
\end{equation}
Note that $V_2'$ intersects $\ell$ only at $Q$. Thus, $|\V(F) \cap V_2' \cap \ell|=1$. We have, 
\begin{align*}
|\V(F) \cap V_2'| &= \sum_{i=0}^{q^2} \left(|\V(F) \cap V_2' \cap \Pi_i| - |\V(F) \cap V_2' \cap \ell| \right) + |\V(F) \cap V_2' \cap \ell| \\
&\le \left((q^2 + 1) - 1\right) + \sum_{i=1}^{q^2} \left( d(q+1) - 1 \right) + 1 \\
&= d(q+1)q^2 + 1.
\end{align*}
The inequality above, of course, follows from Equations \eqref{eq31} and \eqref{eq32}. 

\textbf{Case 2: $V(F)$ contains a plane \textit{not} passing through $P$.} Suppose $V(F)$ contains a plane $\Pi$ not passing through $P$. Then $V_2'$ intersects $\Pi$ at a nonsingular Hermitian curve $V_1$. Since $\Pi \subset \V(F)$, we have $\V(F) \cap V_2' \cap \Pi = V_1$. Note that, 
\begin{align*}
|\V(F) \cap V_2'| & = |\V(F) \cap V_2' \cap \Pi| + |\V(F) \cap V_2' \cap \Pi^C| \\
& = |V_1| + |\V(F) \cap V_2' \cap \Pi^C| \\
& \le q^3 + 1 + (d-1)(q+1)q^2 \\
&< d(q+1)q^2 + 1.
\end{align*}
The inequality above follows from Lemma \ref{lem32}. 
\end{proof}

\begin{remark}\normalfont
It is not difficult to produce a surface of degree $d$ that intersects $V_2'$ at exactly $d(q+1)q^2 + 1$ many $\Fqt$-rational points.  Take a plane $\Pi$ defined over $\Fqt$, not passing through $P$. As noted above, the plane $\Pi$ intersects $V_2'$ at a nonsingular Hermitian curve $V_1$. Take a plane curve $C$ defined over $\Fqt$ of degree $d$ in $\Pi$ that intersects $V_1$ at exactly $d(q+1)$ many $\Fqt$-rational points. It is now evident that the cone over $C$, which is a surface of degree $d$ defined over $\Fqt$ with center at $P$ intersects $V_2'$ at exactly $d(q+1)q^2 + 1$ many $\Fqt$-rational points. While our argument shows that the bound of the above theorem is strict, it does not classify all the surfaces of degree $d$ that attain this bound. This problem could be interesting for further investigation. 
\end{remark}

\section{General results towards Edoukou's conjecture}\label{sec:gen}

Throughout this section, we will denote by $\X$, a threefold of degree $d$ defined over $\Fqt$ contained in $\PP^4$. That is $\X = V(F)$, the set of all points $P \in \PP^4$ such that $F(P) = 0$, for some homogeneous polynomial $F \in \Fqt [x_0, x_1, x_2, x_3, x_4]$ with $\deg F = d$. Since the linear sections of Hermitian threefolds are well-understood, we will be assuming that $1 < d \le q$. Also, as before, the non-degenerate Hermitian threefold will be denoted by $V_3$. 
\begin{proposition}\label{nogen}
    Suppose that $\X$ does not contain a generator. Then 
    $$|\X (\Fqt) \cap V_3| \le d(q^5 + 1).$$ 
    In particular, we have $|\X (\Fqt) \cap V_3| < d(q^5 + q^2) + q^3 + 1.$
\end{proposition}

\begin{proof}
    Let us consider the incidence set 
    $$I = \{ (P, \ell) : P \in \X(\Fqt) \cap V_3, \ell \ \text{is a generator}, P \in \ell\}.$$
    We will count $|I|$ in two ways. First, we see that,
    \begin{equation}\label{eq1}
    |I| = \sum_{P \in \X(\Fqt) \cap V_3} \# \{\ell : P \in \ell \} 
    = |\X(\Fqt) \cap V_3| (q^3 + 1).
    \end{equation}
    The fact that $\# \{\ell : P \in \ell, \ell \ \text{is a generator of } \ V_3\} = q^3 + 1$ is already noted in Remark \ref{rem0}. On the other hand, using the fact that $V_3$ contains $(q^5 + 1) (q^3 + 1)$ generators, we have
    \begin{equation}\label{eq2}
    |I| = \sum_{\ell \subset V_3} |\X (\Fqt) \cap \ell| \le d (q^5 + 1) (q^3 + 1).
    \end{equation}
    It follows from Equations \eqref{eq1} and \eqref{eq2} that 
    $|\X(\Fqt) \cap V_3| \le d (q^5 + 1)$.
\end{proof}

It is now natural to proceed to determine an upper bound for $|\X(\Fqt) \cap V_3|$ in the case when $\X$ contains a generator of $V_3$, but no planes in $\PP^4$ defined over $\Fqt$. However, we do not have a good bound in the general case and as a consequence, the complete solution to Conjecture \ref{EdC} remains open. We will present some partial results under this hypothesis later in this section and  we will study this case for cubic threefolds in the subsequent section. For the time being, let us deal with the case when $\X$ contains a plane in $\PP^4$ defined over $\Fqt$, but it does not contain a hyperplane in $\PP^4$ defined over $\Fqt$. 

\begin{proposition}\label{plane}
    Suppose $\X$ contains a plane defined over $\Fqt$, but no hyperplane defined over $\Fqt$. Then 
    $$|\X(\Fqt) \cap V_3| \le (d-1) q^5 + (d-1)q^4 + dq^3 + dq^2 + 1.$$
    In particular, $|\X(\Fqt) \cap V_3| < d(q^5 + q^2) + q^3 + 1.$
\end{proposition}

\begin{proof}
    Let $\Pi$ be a plane defined over $\Fqt$ contained in $\X$. Let us denote
    $$\B'(\Pi) := \{\Sigma: \Pi \subset \Sigma, \Sigma \ \text{is a hyperplane defined over} \ \Fqt \subset \PP^4\}.$$
    Note that for any plane $\Pi \subset \PP^4$ defined over $\Fqt$, we have
    $|V_3 \cap \Pi| \le q^3 + q^2 + 1.$
    Also, for any $\Sigma \in \B' (\Pi)$, using Proposition \ref{lac}, and the fact that $\Sigma \not\subset \X$, we may conclude that $$|\X (\Fqt) \cap V_3 \cap (\Sigma \setminus \Pi)| \le (d-1)(q+1)q^2.$$ 
    We have,
    \begin{align*}
      |\X(\Fqt) \cap V_3|  &= |\X(\Fqt) \cap V_3 \cap \Pi| + \sum_{\Sigma \in \B' (\Pi)} |\X(\Fqt) \cap V_3 \cap (\Sigma \setminus \Pi)| \\
      &\le q^3 + q^2 + 1 + (q^2 + 1) (d-1)(q+1)q^2 \\
      &= (d-1)q^5 + (d-1)q^4 + dq^3 + dq^2 + 1.
    \end{align*}
The last inequality follows trivially since $d \le q$.    
\end{proof}

As mentioned above, it is difficult to find good bounds for $|\X (\Fqt) \cap V_3|$ when $\X$ contains a generator of $V_3$ but does not contain a plane defined over $\Fqt$. Some partial results to this end are given below. 

\begin{proposition}\label{dgen}
    Suppose that $\X$ contains no plane in $\PP^4$ defined over $\Fqt$. If there exists a plane $\Pi_0$ in $\PP^4$ defined over $\Fqt$ such that $\Pi_0 \cap \X$ has $d$ generators of $V_3$, then 
    $$|\X(\Fqt) \cap V_3|\le dq^5 - q^4 + dq^3 + dq^2 + 1< d(q^5 + q^2) + q^3 + 1.$$
\end{proposition}

\begin{proof}
    Let $\Pi_0 \cap V_3$ contain $d$ generators of $V_3$. Since a plane containing a generator of $V_3$ can intersect $V_3$ at either one line or $q+1$ concurrent lines, it follows that $V_3 \cap \Pi_0$  is a union of $q+1$ lines $\ell_0, \dots, \ell_q$ passing through a common point, say $P \in V_3$. It is also easy to see that $\B' (\Pi_0)$ contains exactly one hyperplane that is tangent to $V_3$, namely, the tangent hyperplane to $V_3$ at $P$. Indeed, if $\B'(\Pi)$ contains a tangent hyperplane $T_Q(V_3)$ where $Q \neq P$, then $\Pi_0 = T_P(V_3) \cap T_Q(V_3)$ and hence the lines $\ell_0, \dots, \ell_q$ are contained in $T_Q(V_3)$. Consequently, the $q+1$ lines meet at the point $Q$, implying $P = Q$, a contradiction. \ 
    If $\Sigma = T_P (V_3)$, then $\X \cap \Sigma$ is a surface of degree $d$ in $\Sigma$. It now follows from Theorem \ref{degen} that
     $$|\X(\Fqt) \cap V_3 \cap T_P (V_3)| \le d(q+1)q^2 + 1.$$ 
     If $\Sigma \in \B' (\Pi_0)$, but $\Sigma \neq T_P (V_3)$, we conclude from Theorem \ref{hyperplane} (b), that $\Sigma \cap V_3$ is a non-degenerate Hermitian surface in $\Sigma$. Now Theorem \ref{SoB} yields 
    $$|\Sigma \cap \X(\Fqt)  \cap V_3| \le d(q^3 + q^2 - q) + q + 1.$$
    However, the upper bound above is attained if and only if $\Sigma \cap \X$ contains $d$ planes that are tangent to $V_3$. Since, in our case, the threefold $\X$ does not contain a plane, a strict bound \cite[Theorem 5.1]{BDH} applies and we have 
    $$|\Sigma \cap \X(\Fqt)  \cap V_3| \le dq^3 + (d-1)q^2 + 1.$$
    Finally, $\X(\Fqt)$ does not contain $\Pi_0$ and $\X \cap \Pi_0$ contains the union of $d$ generators. It follows that $\X \cap \Pi_0$ is exactly the union of $d$ generators of $V_3$. Thus
    $$|\X(\Fqt) \cap V_3 \cap \Pi_0| = dq^2 + 1.$$ 
    We may thus conclude that, if $\Sigma \in \B'(\Pi_0)$ and $\Sigma \neq T_P (V_3)$, then 
    $$|\X(\Fqt) \cap V_3 \cap (\Sigma \setminus \Pi_0)| \le dq^3 + (d-1)q^2 + 1 - (dq^2 + 1) = dq^3 - q^2.$$
    Following the same argument as in Proposition \ref{plane}, we see that
    \begin{align*}
       |\X(\Fqt) \cap V_3|  
       &=  |\X(\Fqt) \cap V_3 \cap T_P (V_3) | + \sum_{\substack{\Sigma \in \B' (\Pi_0) \\ \Sigma \neq T_P (V_3)}} |\X(\Fqt) \cap V_3 \cap (\Sigma \setminus \Pi_0)| \\
       &\le d(q+1)q^2 + 1 +  q^2 (d q^3 - q^2) \\
       &= dq^5 - q^4 + dq^3 + dq^2 + 1
    \end{align*}
    The last strict inequality in the assertion follows trivially since $d \le q$.  
\end{proof}

\begin{proposition}\label{atmost1gen}
   Suppose that $\X$ contains no plane in $\PP^4$ defined over $\Fqt$. Furthermore, assume that $\X \cap \Pi$ contains at most one generator for any plane $\Pi$ defined over $\Fqt$ in $\PP^4$. Then 
    $$|\X(\Fqt) \cap V_3| \le (d-1)q^5 + q^2 + (d-1)q^3 + (d-1)q + 1< d(q^5 + q^2) + q^3 + 1.$$
\end{proposition}

\begin{proof}
    Suppose that for all plane $\Pi$ defined over $\Fqt$ in $\PP^4$, $\X \cap \Pi$ does not contain any generator. This in particular implies that $\X$ does not contain any generator. Then Proposition \ref{nogen} applies and proves the assertion. Thus we may assume that $\X$ contains a generator $\ell$ of $V_3$.

 \noindent   \textbf{Claim:} For every $\Pi \in \B (\ell)$, we have $|\X \cap V_3 \cap (\Pi \setminus \ell)| \le (d-1)q$. 

\noindent    \textit{Proof of claim:} Note that, for any $\Pi \in \B (\ell)$ the intersection
 $\Pi \cap V_3$ consists of at most $q+1$ generators. Since $\X \cap V_3$ contains no generators of $V_3$ other than $\ell$, it follows that $\X \cap (\Pi \setminus \ell)$ and $V_3 \cap (\Pi \setminus \ell)$ are plane curves of degrees $d-1$ and at most $q$ respectively. Furthermore, the two affine curves have no common components. Using B\'{e}zout's Theorem, we now conclude that 
        $|\X \cap V_3 \cap (\Pi \setminus \ell)| = |\left(\X \cap (\Pi \setminus \ell)\right) \cap \left(V_3 \cap (\Pi \setminus \ell) \right)| \le (d-1)q.$ 
Thus
\begin{align*}
    |\X(\Fqt) \cap V_3| &= |\ell (\Fqt)| + \sum_{\Pi \in \B (\ell)}|\X(\Fqt) \cap V_3 \cap (\Pi \setminus \ell)| \\
    &\le (q^2 + 1) + (q^4 + q^2 + 1) (d-1)q  \\
    &= (d-1)q^5 + q^2 + (d-1)q^3 + (d-1)q + 1 \\
    &< d(q^5 + q^2) + q^3 + 1.
\end{align*}
The first inequality follows from the claim, whereas the strict inequality is an immediate consequence of the hypothesis $d \le q$. 
\end{proof}

Before proceeding further, we show that the results derived above are enough to prove Edoukou's bound. To this end, let us first prove the following Lemma. 

\begin{lemma}\label{ed2}
    If $\Sigma_1$ be a hyperplane in $\PP^4$ defined over $\Fqt$ that is not a tangent to $V_3$, then for any hyperplane $\Sigma$ in $\PP^4$ defined over $\Fqt$, we have $|\Sigma \cap \Sigma_1 \cap V_3| \ge q^3 + 1$. 
\end{lemma}

\begin{proof}
    If $\Sigma_1 = \Sigma$, then $|\Sigma \cap \Sigma_1 \cap V_3| = |\Sigma_1 \cap V_3| = q^5 +q^3+ q^2 + 1 > q^3 + 1$. We may thus assume that $\Sigma \neq \Sigma_1$. Thus $\Sigma \cap \Sigma_1$ is a plane contained in the projective space $\Sigma_1$. Since $\Sigma_1$ is not a tangent to $V_3$, we see that $V_3 \cap \Sigma_1$ is a non-degenerate Hermitian surface in $\Sigma_1$. Note that $\Sigma \cap \Sigma_1$ is a plane defined over $\Fqt$ contained in $\Sigma_1$. We know that a planar section of a nondegenerate Hermitian surface is either a plane nonsingular Hermitian curve or a union of $q+1$ lines passing through a common point. Thus $|\Sigma \cap \Sigma_1 \cap V_3| \ge q^3 + 1$, proving the assertion. 
\end{proof}

Now, we are ready to prove Edoukou's result on the intersection of a non-degenerate Hermitian threefold and a quadric threefold in $\PP^4$.

\begin{corollary}[Edoukou's Theorem]\cite[Theorem 4.1]{E1}\label{Eprove}
   Let $Q$ be a quadric threefold defined over $\Fqt$ in $\PP^4$. Then 
   $$|V_3 \cap Q (\Fqt)| \le 2(q^5 + q^2) + q^3 + 1.$$
   The bound is attained by a quadric $Q$ defined over $\Fqt$ if and only if $Q$ is a union of two hyperplanes $\Sigma_1$ and $\Sigma_2$ defined over $\Fqt$, both not tangents to $V_3$, such that the plane $\Sigma_1 \cap \Sigma_2$ intersects $V_3$ at a non-degenerate plane Hermitian curve. 
\end{corollary}

\begin{proof}
    We distinguish the proof into two cases:
    
        {\bf Case 1:} Suppose $Q$ is reducible. Then $Q$ must be a union of hyperplanes $\Sigma_1$ and $\Sigma_2$ defined over $\Fqt$. From Theorem \ref{hyperplane}, we see that for any hyperplane $\Sigma$ in $\PP^4$ defined over $\Fqt$,
        $$|\Sigma (\Fqt) \cap V_3| 
        = \begin{cases}
            q^5 + q^3 + q^2 + 1 \ \ \text{if} \ \Sigma \ \text{is not a tangent to} \ V_3 \\
            q^5 + q^2 + 1 \ \ \text{otherwise}
        \end{cases}.$$
        If $\Sigma_1, \Sigma_2$ are both tangents to $V_3$, then 
        $$|V_3 \cap Q(\Fqt)| \le |V_3 \cap \Sigma_1| + |V_3 \cap \Sigma_2| \le 2(q^5 + q^2 + 1) < 2(q^5 + q^2) + q^3 + 1.$$
        So we may assume without loss of generality that $\Sigma_1$ is not a tangent to $V_3$. Thus 
        \begin{align*}
        |V_3 \cap Q(\Fqt)| &= |V_3 \cap \Sigma_1| + |V_3 \cap \Sigma_2| - |V_3 \cap \Sigma_1 \cap \Sigma_2| \\
        &\le 2(q^5 + q^3 + q^2 + 1) - (q^3 + 1) \\
        &= 2(q^5 + q^2) + q^3 + 1.
        \end{align*}
        The inequality above is a consequence of Lemma \ref{ed2}. Note that here equality occurs if and only if both $\Sigma_1$ and $\Sigma_2$ are not tangents to $V_3$ and the plane $\Sigma_1 \cap \Sigma_2$ intersects $V_3$ at a nonsingular plane Hermitian curve. 
        
        \textbf{Case 2:} Suppose  $Q$ is irreducible. Thus $Q$ can not contain a hyperplane defined over $\Fqt$. If $Q$ contains a plane defined over $\Fqt$, Proposition \ref{plane} applies and establishes the assertion. We may assume that $Q$ contains no plane defined over $\Fqt$. If $Q$ contains no generators of $V_3$, then the assertion follows from Proposition \ref{nogen}. Suppose $Q$ contains a generator of $V_3$. If every plane in $\PP^4$ defined over $\Fqt$ contains at most one generator of $V_3$, then Proposition \ref{atmost1gen} applies. Finally, if a plane in $\PP^4$ defined over $\Fqt$ that contains two generators of $V_3$ exists, then the assertion inequality is obtained from Proposition \ref{dgen}. We note that all the inequalities in this case are strict.  This completes the proof. 
\end{proof}

We conclude this section by showing that if $\X$ contains a line $\ell$ defined over $\Fqt$ that is tangent to $V_3$, then $|\X (\Fq) \cap V_3|$ satisfies the bound in Edoukou's conjecture. To this end, we have the following Lemma. 

\begin{lemma}\label{lem1}
    Suppose $P \in V_3$ and $\ell_P$ is a tangent line to $V_3$ passing through $P$. Then
    \begin{enumerate}
        \item[(a)] there exist $q^2 + 1$ planes containing $\ell_P$ that are contained in $T_P (V_3)$, and
        \item[(b)] if $\Pi \in \B(\ell_P)$ and $\Pi \not\subset T_P(V_3)$, then $\Pi \cap V_3$ does not contain any generator of $V_3$.  
    \end{enumerate}
\end{lemma}

\begin{proof}
To prove part (a), it is enough to show that $\ell_P \subset T_P (V_3)$. But the later is true since any line passing through $P$ that is not contained in $T_P(V_3)$ intersects $V_3$ at $q+1$ points. 
Suppose $\Pi \in \B(\ell_P)$ and $\Pi$ contains a generator, say $\ell$, of $V_3$. Now $\ell$ and $\ell_P$ are two distinct lines in $\Pi$ and hence $\Pi$ is the unique plane containing $\ell$ and $\ell_P$. On the other hand, both the lines $\ell$ and $\ell_P$ are contained in $T_P(V_3)$. It follows that $\Pi \subset T_P(V_3)$.
This proves part (b). 
\end{proof}

\begin{proposition}\label{tanline}
    Suppose that $\X$ contains a tangent line to $V_3$ and $\X$ does not contain any plane defined over $\Fqt$. Then $$|\X(\Fqt) \cap V_3| \le (d-1)q^5 + (d-1)q^4 + dq^3 + dq^2 + 1 < d(q^5 + q^2) + q^3 + 1.$$ 
\end{proposition}

\begin{proof}
    Suppose $\X$ contains a line $\ell_P$ that is tangent to $V_3$ at the point $P \in V_3$. Write
    $$\B (\ell_P) =\{\Pi_0, \dots, \Pi_{q^2}, \Pi_{q^2+1}, \dots, \Pi_{q^4 + q^2}\},$$
    where $\Pi_0, \dots, \Pi_{q^2} \subset T_P (V_3)$. Note that
    \begin{equation}\label{one}
        |\ell_P \cap \X(\Fqt) \cap V_3| + \sum_{i=0}^{q^2} 
|\X(\Fqt) \cap V_3 \cap (\Pi \setminus \ell_P)| = |\X(\Fqt) \cap V_3 \cap T_P (V_3)| \le d (q+1)q^2 + 1.
    \end{equation}
    Note that the last inequality follows from Theorem \ref{degen}. Using Lemma \ref{lem1}, we conclude that, for $j = q^2 + 1, \dots, q^4 + q^2$, the plane $\Pi_j$ intersects $V_3$ at a plane non-degenerate Hermitian curve. Also the curve $\X \cap (\Pi_j \setminus \ell_P)$ is a plane curve of degree at most $d-1$. 
    From B\'{e}zout's theorem, 
    \begin{equation}\label{two}
    |\X(\Fqt) \cap V_3 \cap (\Pi \setminus \ell_P)| \le (d-1)(q+1).
    \end{equation}
    Finally, using equations \eqref{one} and \eqref{two}, we see that,
    \begin{align*}
        |\X(\Fqt) \cap V_3| &= |\ell_P \cap \X (\Fqt) \cap V_3| + \sum_{j=0}^{q^4 + q^2} |\X(\Fqt) \cap V_3 \cap (\Pi \setminus \ell_P)| \\
        &=|\X(\Fqt) \cap V_3 \cap T_P (V_3)|+ \sum_{j=q^2 + 1}^{q^4 + q^2} |\X(\Fqt) \cap V_3 \cap (\Pi \setminus \ell_P)| \\
        &\le d(q+1)q^2 + 1 + q^4 (d-1)(q+1) \\
        &= (d-1)q^5 + (d-1)q^4 + dq^3 + dq^2 + 1 < d(q^5 + q^2) + q^3 + 1. 
    \end{align*}
    This completes the proof.
\end{proof}

\section{Intersection of cubic threefolds with Hermitian threefold}\label{sec:cubic}

The main goal of this section is to determine the maximum number of $\Fqt$-rational points on the intersection of a cubic threefold and Hermitian threefolds in $\PP^4$. To this end, we will denote by $\C$ a cubic threefold in $\PP^4$ defined over $\Fqt$ while we keep using $V_3$ for denoting the non-degenerate Hermitian threefold in $\PP^4$. Note that, we have dealt with threefolds that  
\begin{enumerate}
    \item[(a)] contains no generators of $V_3$ and
    \item[(b)] contains a plane defined over $\Fqt$, but no hyperplanes in $\PP^4$ defined over $\Fqt$
\end{enumerate}
in the previous section. Thus, we will need to study the cases 
when 
\begin{enumerate}
    \item[(i)] $\C$ contains generators of $V_3$ but no planes defined over $\Fqt$ in $\PP^4$ and 
    \item[(ii)] $\C$ contains a hyperplane of $\PP^4$ defined over $\Fqt$
\end{enumerate}
in this section. For ease of reading, we will divide the section into two subsections. 
\subsection{Cubic threefolds containing a hyperplane}
In this subsection, we will make use of Edoukou's result (Proposition \ref{ed1}) to show that if $\C$ contains a hyperplane in $\PP^4$ defined over $\Fqt$, then the bound in Edoukou's conjecture holds in the affirmative.  
\begin{proposition}\label{reducible}
    Let $\mathcal{C}$ be a cubic threefold in $\PP^4$ defined over $\Fqt$ that contains a hyperplane $\Sigma$ defined over $\Fqt$. Then
    $$|\mathcal{C}(\Fqt) \cap V_3| \le 3(q^5 + q^2) + q^3 + 1.$$
\end{proposition}

\begin{proof}
  Suppose that $\mathcal{C}(\Fqt)$ contains a hyperplane $\Sigma$ defined over $\Fqt$. It follows immediately that $\mathcal{C} = \Sigma \cup Q$, where $Q$ is a quadric threefold defined over $\Fqt$. Now, 
  $|\Sigma \cap V_3| \le q^5 + q^3 + q^2 + 1$. We distinguish four cases.

\textbf{Case 1: $Q$ is of Type I.} Then $Q = \Sigma_1 \cup \Sigma_2$, where $\Sigma_1$ and $\Sigma_2$ are hyperplanes in $\PP^4$ defined over $\Fqt$, neither of them tangent to $V_3$. Using Proposition \ref{ed1} we derive that $|Q \cap V_3| = 2(q^5 + q^2) + q^3 + 1$. We have, 
\begin{align*}
    |\mathcal{C} \cap V_3| &= |\Sigma \cap V_3| + |Q(\Fqt) \cap V_3| - |(\Sigma_1 \cup \Sigma_2) \cap \Sigma \cap V_3| \\ 
    & \le |\Sigma \cap V_3| + |Q(\Fqt) \cap V_3| - |(\Sigma_1  \cap \Sigma) \cap V_3| \\ 
    & \le (q^5 + q^3 + q^2 + 1) + 2(q^5 + q^2) + q^3 + 1 - (q^3 + 1) \\
    &= 3(q^5 + q^2) + q^3 + 1. 
\end{align*}
The last inequality above follows from Lemma \ref{ed2}.

\textbf{Case 2: $Q$ is of Type II.} We have $Q = \Sigma_1 \cup \Sigma_2$, where $\Sigma_1$ and $\Sigma_2$ are hyperplanes in $\PP^4$ defined over $\Fqt$, neither of them tangent to $V_3$. Using Proposition \ref{ed1} we derive that $|Q(\Fqt) \cap V_3| = 2q^5 + q^2 + q^3 + 1$. We have, 
\begin{align*}
    |\mathcal{C}(\Fqt) \cap V_3| &= |\Sigma \cap V_3| + |Q(\Fqt) \cap V_3| - |(\Sigma_1 \cup \Sigma_2) \cap \Sigma \cap V_3| \\ 
    & \le |\Sigma \cap V_3| + |Q(\Fqt) \cap V_3| - |(\Sigma_1  \cap \Sigma) \cap V_3| \\ 
    & \le (q^5 + q^3 + q^2 + 1) + 2q^5 + q^2 + q^3 + 1 - (q^3 + 1) \\
    &= 3q^5 + q^3 + 2q^2 + 1 < 3(q^5 + q^2) + q^3 + 1. 
\end{align*}
The last inequality above follows from Lemma \ref{ed2}.

\textbf{Case 3: $Q$ is of Type III.} Here $Q = \Sigma_1 \cup \Sigma_2$, where $\Sigma_1$ and $\Sigma_2$ are hyperplanes in $\PP^4$ defined over $\Fqt$. Furthermore, the hyperplane $\Sigma_2$ is not a tangent to $V_3$. Using Proposition \ref{ed1} we derive that $|Q(\Fqt) \cap V_3| = 2q^5 + 2q^2 + q^3 + 1$. We have, 
\begin{align*}
    |\mathcal{C}(\Fqt) \cap V_3| &= |\Sigma \cap V_3| + |Q(\Fqt) \cap V_3| - |(\Sigma_1 \cup \Sigma_2) \cap \Sigma \cap V_3| \\ 
    & \le |\Sigma \cap V_3| + |Q(\Fqt) \cap V_3| - |(\Sigma_2  \cap \Sigma) \cap V_3| \\ 
    & \le (q^5 + q^3 + q^2 + 1) + 2q^5 + 2q^2 + 1 - (q^3 + 1) \\
    &= 3q^5 + 3q^2 + 1 < 3(q^5 + q^2) + q^3+ 1. 
\end{align*}
\textbf{Case 4: $Q$ is not of Type I, Type II or Type III.} In this case, from Proposition \ref{ed1}, we obtain that $|Q \cap V_3| \le 2q^5 + q^2 + 1$. Thus 
\begin{align*}
    |\mathcal{C} \cap V_3| & \le |\Sigma \cap V_3| + |Q \cap V_3| \\
    &\le q^5 + q^3 + q^2 + 1 + 2q^5 + q^2 + 1 \\
    &= 3q^5 + 2q^2 + q^3 + 2 < 3(q^5 + q^2) + q^3 + 1.
\end{align*}
This completes the proof. 
\end{proof}

\begin{remark} \label{attained}\normalfont
For the sake of proving the inequality in the above Proposition, it is not really necessary to distinguish Case 1, Case 2, and Case 3. In each of these three cases, the presence of one nontangent hyperplane in the conic plays a pivotal role  in establishing the required upper bound. However, the above distinction shows us that the upper bound could be attained by a cubic threefold $\C$ in $\PP^4$ defined over $\Fqt$ and containing a hyperplane in $\PP^4$ defined over $\Fqt$ will attain the upper bound if and only if the following conditions are satisfied:
\begin{enumerate}
\item[(a)] $\Sigma_1, \Sigma_2, \Sigma_3$ are distinct hyperplanes in $\PP^4$ defined over $\Fqt$.
\item[(b)] $\Sigma_1, \Sigma_2, \Sigma_3$ intersect at a common plane $\Pi$ in $\PP^4$ such that $\Pi \cap V_3$ is a plane nonsingular Hermitian curve. 
\end{enumerate}
In fact, we will establish at the end of this article that if $q \ge 7$, then a cubic threefold $\C$ defined over $\Fqt$ in $\PP^4$ attains the upper bound of Edoukou's conjecture if and only if $\C$ satisfies the conditions mentioned above. 
\end{remark}

\subsection{Cubic threefolds containing no hyperplanes}

\begin{lemma}\label{ALR}
    Let $\C'$ be an irreducible plane cubic curve defined over $\Fqt$. Then
    \begin{enumerate}
        \item[(a)] If $\C'$ is absolutely irreducible, then 
        $$||\C'(\Fqt)| - (q^2 + 1)| \le 2q.$$
        \item[(b)] If $\C'$ is not absolutely irreducible, then $|\C' (\Fqt)| \le 1$. 
    \end{enumerate}
\end{lemma}

\begin{proof}
\
    \begin{enumerate}
        \item[(a)] Follows from Theorem \ref{ap}.
        \item[(b)] Suppose that $\C'$ is a plane cubic curve defined over $\Fqt$ that is irreducible over $\Fqt$ but reducible in $\overline{\FF}_q$, the algebraic closure of $\Fq$. From \cite[Theorem 3.2]{LR}, we see that if $Z$ is an absolutely irreducible component of $\C'$, then $\deg (Z) \mid \deg \C'$. Since $\deg \C' = 3$, it follows that $\C'$ is a union of three lines passing through a common point, say $P$. Thus Proposition \ref{lac1} implies that $\C' (\Fqt) \subseteq \{P\}$ and hence the assertion is established. 
        \end{enumerate}
    This completes the proof.
\end{proof}

\begin{lemma}\label{absirr}
Let $q \ge 7$ and $\C$ be a cubic threefold defined over $\Fqt$ and $P \in \C (\Fqt) \cap V_3$. If $\C \cap T_P(V_3)$ is a cone over a plane cubic absolutely irreducible curve $\C^*$ with center at $P$, then $\C$ contains a line that is tangent to $V_3$ at $P$. 
\end{lemma}

\begin{proof}
    Since $\C^*$ is a plane absolutely irreducible curve defined over $\Fqt$, it follows from Lemma \ref{ALR} that $|\C^* (\Fqt)| \ge q^2 + 1 - 2q$. Suppose that $\C^* \subset \Pi$, where $\Pi$ is a plane, not passing through $P$, defined over $\Fqt$ which is contained in $T_P(V_3)$. Now $V_3 \cap \Pi$ is a plane non-degenerate Hermitian curve and hence B\'{e}zout's theorem implies 
    $|(V_3 \cap \Pi) \cap C^*| \le 3(q+1)$. Since $q \ge 7$, it follows that $|\C^* (\Fqt)| > |(V_3 \cap \Pi) \cap C^*|$ and hence there is a point $P' \in \C^*$ that is not contained in $V_3$. Since $\C$ is a cone over $C^*$ with center at $P$, the line $\ell$ joining $P$ and $P'$ is contained in $\C$. Since $\ell$ is a line defined over $\Fqt$ passing through $P$, the line $\ell$ must be either a tangent line or a generator. Since $P' \not\in V_3$, it follows that $\ell$ can not be a generator and the assertion follows. 
\end{proof}

\begin{lemma}\label{lem:5.6}
Let $P \in V_3$ and $\ell$ be a generator of $V_3$ containing $P$. Let $\C$ be a cubic threefold defined over $\Fqt$ such that for any plane $\Pi$ defined over $\Fqt$, the threefold $\C$ contains at most two generators of $V_3$. If $\ell$ is contained in $\C$, then one of the following holds:
\begin{enumerate}
\item[(a)] $T_P(V_3) \cap \C$ contains a cone over a plane curve with center at $P$,
\item[(b)] $|T_P(V_3) \cap \C \cap V_3| \le 2q^3 + 6q^2 - 3q -4$.
\end{enumerate}
\end{lemma}

\begin{proof}
    Suppose (a) does not hold. Then, in view of Theorem \ref{D}, we see that $T_P (V_3) \cap \C$ contains at most six lines passing though $P$. Let
    $$\mathcal{S}:=  \{\Pi \in \B_{T_P(V_3)} (\ell)  : \C \cap \Pi \ \text{contains  two generators of} \ V_3\} \ \ \text{and} \ \ s = |\mathcal{S}|.$$
    It follows that $s \le 5$ and for any $\Pi \in \mathcal{S}$, we have $|\Pi \cap V_3 \cap (\C \setminus \ell)| \le q^2 + q - 1$. 
    On the other hand, for any $\Pi \in \B_{T_P(V_3)} (\ell) \setminus \mathcal{S}$, using B\'{e}zout's Theorem, we see that $|\Pi \cap V_3 \cap (\C \setminus \ell)| \le 2q$. Thus
    \begin{align*}
        |T_P (V_3) \cap C \cap V_3| &\le |\ell| + \sum_{\Pi \in \mathcal{S}} |\Pi \cap V_3 \cap (\C \setminus \ell)| + \sum_{\Pi \in \B_{T_P(V_3)} (\ell) \setminus \mathcal{S}} |\Pi \cap V_3 \cap (\C \setminus \ell)| \\
        & \le (q^2 + 1) + s (q^2 + q - 1) + (q^2 + 1 - s) (2q) \\
        & = 2q^3 + q^2 + 2q + 1 + s (q^2 - q - 1) \\
        & \le 2q^3 + q^2 + 2q + 1 + 5 (q^2 - q - 1)  \ \ (\text{since} \ q^2 - q - 1 > 0) \\
        & = 2q^3 + 6q^2 - 3q - 4. 
    \end{align*}
    This completes the proof. 
\end{proof}
\begin{proposition}\label{casetwo}
Let $q \ge 7$. Let $\C$ be a cubic threefold defined over $\Fqt$ that contains a generator of $V_3$ but contains no plane defined over $\Fqt$. Moreover, assume that for any plane $\Pi$ defined over $\Fqt$, the intersection $\C \cap \Pi$ contains at most two generators of $V_3$. Then $$|\C(\Fqt) \cap V_3| < 3(q^5 + q^2) + q^3 + 1.$$
\end{proposition}

\begin{proof}
   Suppose $\C$ contains a generator $\ell$ of $V_3$. Now, as discussed in Remark \ref{remD}, for every $P \in \ell$, either the cubic surface $\C \cap T_P(V_3)$ is  a cone over a plane cubic curve $\C_P'$ with center at $P$ or $\C \cap T_P(V_3)$ contains at most $6$ lines passing through $P$. 

   Let $L = \{ P \in \ell : \C \cap T_P(V_3) \ \text{is a cone over a plane cubic} \ \C'_P \ \text{with center at} \ P \}.$ 
   Observe that for each $P \in L$, if $\C'_P$ is reducible, then $\C'_P$ contains a line defined over $\Fqt$, and as a consequence, $\C$ contains a plane defined over $\Fqt$. We may thus assume that for each $P \in L$, the plane curve $\C'_P$ is irreducible. If $\C'_P$ is absolutely irreducible, then it follows from Lemma \ref{absirr},  that $\C$ contains a line that is tangent to $V_3$. Here, Proposition \ref{tanline} applies, and we obtain the assertion. We may thus assume from now that for every $P \in L$, the plane curve $\C'_P$ is irreducible but not absolutely irreducible. It now follows from Lemma \ref{ALR} (b), that for each $P \in L$, we have $T_P(V_3) \cap \C \cap V_3 = \ell.$ Consequently, we have the following:
   \begin{equation} \label{formula}
    |V_3 \cap T_P(V_3) \cap (\C \setminus \ell)| = 
    \begin{cases}
        0 \ \ &\text{if} \ P \in L \\
        \le 2q^3 + 5q^2 - 3q -5 \ \ &\text{if} \ P \in \ell \setminus L.
    \end{cases}
    \end{equation}
   The last inequality follows from Lemma \ref{lem:5.6}. Since
   $\PP^4 = \displaystyle{\bigcup_{P \in \ell}} T_P (V_3)$, we have
   \begin{align*}
     |C(\Fqt) \cap V_3| & = |\ell| + \sum_{P \in L}  |V_3 \cap T_P(V_3) \cap (\C \setminus \ell)| + \sum_{P \in \ell \setminus L}|V_3 \cap T_P(V_3) \cap (\C \setminus \ell)| \\
     &\le (q^2 + 1) + |L| \cdot 0  + (q^2 + 1 - |L|) (2q^3 + 5q^2 - 3q -5) \\
     &= q^2 + 1 + 2q^5 + 5q^4 - q^3 - 3q - 5 - |L|(2q^3 + 5q^2 - 3q -5) \\
     &\le 2q^5 + 5q^4 - q^3 + q^2 - 3q - 4 \\
     &< 3(q^5 + q^2) + q^3 + 1.
   \end{align*}
 The last inequality is true since $q \ge 7$, as per hypothesis. This completes the proof.  
\end{proof}

\begin{theorem}
    Let $F \in \Fqt [x_0, \dots, x_4]$ be a homogeneous polynomial of degree three and $V_3$ be a non-degenerate Hermitian threefold in $\PP^4 (\Fqt)$. If $q\ge7$, then 
$$|V(F)(\Fqt) \cap V_3| \le 3 (q^5 + q^2) + q^3 + 1.$$
Moreover, this bound is attained by a cubic threefold $V(F)$ in $\PP^4$, if and only if $V(F)$ is a union of three non-tangent hyperplanes $\Sigma_1, \Sigma_2$ and $\Sigma_3$ defined over $\Fqt$, such that each of the hyperplanes contains a common plane $\Pi$ defined over $\Fqt$ and $\Pi$ intersects $V_3$ at a non-degenerate plane Hermitian curve.
\end{theorem}

\begin{proof}
Throughout, we will denote by $\C$ the threefold given by the vanishing locus of $F$, that is $\C = V(F)$. We distinguish the proof into two cases. 

 {\bf Case 1:} \textit{$\C$ is reducible over $\Fqt$}. In this case, the threefold $\C$ contains a hyperplane defined over $\Fqt$. From Proposition \ref{reducible} we see that $|V(F)(\Fqt)\cap V_3|\le 3(q^5+q^2)+q^3+1$. Also it follows from Remark \ref{attained} that the bound is attained in this case if and only if $V(F)$ is a union of three non-tangent hyperplanes $\Sigma_1, \Sigma_2$ and $\Sigma_3$ defined over $\Fqt$, such that each of the hyperplanes contains a common plane $\Pi$ defined over $\Fqt$ and $\Pi$ intersects $V_3$ at a non-degenerate plane Hermitian curve.

 {\bf Case 2:} \textit{$\C$ is irreducible over $\Fqt$.} In this case, $\C$ does not contain any hyperplane defined over $\Fqt$. If $\C$ does not contain any generator of $V_3$, then Proposition \ref{nogen} shows that $|\C(\Fqt)\cap V_3|<3(q^5+q^2)+q^3+1$. If $V(F)$ contains a plane defined over $\Fqt$, then Proposition \ref{plane} applies and shows that $|\C(\Fqt)\cap V_3|<3(q^5+q^2)+q^3+1$. So, we may assume that $\C$ contains a generator of $V_3$ but no plane defined over $\Fqt$. It follows that for any plane $\Pi$ defined over $\Fqt$, the plane curve $V(F)\cap \Pi$ can contain at most three generators of $V_3$. If there exists a plane $\Pi_0$ in $\PP^4$ defined over $\Fqt$ such that $V(F)\cap \Pi_0$ has exactly three generators of $V_3$, then Proposition \ref{dgen} implies that $|V(F)(\Fqt)\cap V_3|<3(q^5+q^2)+q^3+1$. Furthermore, if for any plane $\Pi$ defined over $\Fqt$, the intersection $V(F)\cap \Pi$ contains at most two generators of $V_3$, then Proposition \ref{casetwo} shows that $|V(F)(\Fqt)\cap V_3|<3(q^5+q^2)+q^3+1$. Thus, we see that a strict bound always holds in this case. This completes the proof.
\end{proof}

\textbf{Acknowledgements:} We are grateful to the anonymous referees for their important suggestions in making the article better than the previous version.

\end{document}